\documentclass[a4paper,reqno,11pt,oneside]{amsart}
\usepackage{amsmath,geometry}
\usepackage{graphicx}
\usepackage{mathrsfs}
\usepackage{amssymb,amsmath, amsfonts, amsthm,color}
\usepackage{latexsym}
\usepackage{color} 
\usepackage[dvips]{epsfig}
\usepackage[colorlinks]{hyperref}
\hypersetup{
    colorlinks,
    citecolor=blue,
    linkcolor=blue,
    urlcolor=black
}
\usepackage{tikz}
\usepackage{pgfplots}
\usetikzlibrary{patterns}
\usepackage{bm}
\newtheorem{theorem}{Theorem}[section]
\newtheorem{lemma}[theorem]{Lemma}

\newtheorem{proposition}[theorem]{Proposition}

\theoremstyle{definition}

\newtheorem{remark}[theorem]{Remark}

\numberwithin{equation}{section}

\renewcommand{\(}{\left(}
\renewcommand{\)}{\right)}

\begin{document}
\title[Nonlinear Schr\"odinger systems with a large number of components]{Segregated solutions for  nonlinear Schr\"odinger systems with a large number of components}

\author{Haixia Chen}
\address[H. Chen]{
 School of Mathematics and Statistics, Central China Normal University, Wuhan 430079, P. R. China.
}
\email{hxchen@mails.ccnu.edu.cn}

\author{Angela Pistoia}
\address[A. Pistoia]{Dipartimento di Scienze di Base e Applicate per l'Ingegneria, Sapienza Universit\`a di Roma, Via Scarpa 16, 00161 Roma, Italy}
\email{angela.pistoia@uniroma1.it}

\begin{abstract}
In this paper we are concerned with the existence of segregated non-radial solutions for  nonlinear Schr\"odinger systems with a large number of components   in a weak fully attractive or repulsive regime in presence of a suitable  external radial 
potential.\end{abstract}

\date\today
 \thanks{H. Chen is partially supported by the NSFC grants (No.12071169) and the China Scholarship Council (No. 202006770017).  A. Pistoia is also partially supported by INDAM-GNAMPA funds and Fondi di Ateneo ``Sapienza" Universit\`a di Roma (Italy).\\
Data sharing not applicable to this article as no datasets were generated or analysed during the current study.}

\maketitle

\section{Introduction}

The well-known Gross-Pitaevskii system  
\begin{equation}\label{GP}
- \iota \partial_t  \phi_i = \Delta  \phi_i -V_i(x) \phi_i+ \mu_i| \phi_i|^2 \phi_i+  \sum\limits_{j=1\atop j\not=i}^m\beta_{ij}
| \phi_j|^{2}  \phi_i ,\ i=1,\dots,m\end{equation}
 has been proposed as a mathematical model for multispecies
Bose-Einstein condensation in $m$ different  states. We refer to \cite{11,12,M,15,23} for a detailed physical motivation.
Here the complex valued functions $ \phi_i$'s are the wave functions of
the $i-$th condensate, $|\phi_i|$ is the amplitude of the $i$-th density, $ \mu_i$    describes  the interaction between particles of the same 
component
and $\beta_{ij}$, $i\not=j,$  describes   the interaction between particles   of different components, which can be {\em attractive} if $\beta_{ij}>0$ or {\em repulsive} if
  $\beta_{ij}<0$. 
To obtain solitary wave solutions of the Gross-Pitaevskii system \eqref{GP} one sets  $ \phi_i(t,x) = e^{- \iota \lambda_i t} u_i(x)$ and the real functions $u_i$'s solve the system
\begin{equation}\label{S}  -\Delta u_i +\lambda_i u_i +V_i(x) u_i= \mu_i u_i^3+u_i   \sum \limits_{ j=1\atop j\not=i }^m \beta_{ij} u_j^2  \ \hbox{in}\ \mathbb R^n,\ i =1,\dots,m\end{equation}
 where $\mu_i >0$, $\lambda_i>0$, $ \beta_{ij} =\beta_{ji}\in\mathbb R$, $V_i\in C^0(\mathbb R^n)$ and $n\ge2.$
\\

  We are interested in finding solutions whose components  $u_i$ do not identically vanish for any index $i=1,\dots,m.$
\\

In the last decades,   the 
nonlinear Schrödinger system \eqref{S} has been widely studied.   Most of the work has been done  in the autonomous case (i.e.  $V_i$ is constant) and in the sub-critical regime (i.e.  $n=2$ or $n=3$). We refer the reader to a couple of   recent papers   \cite{bmw,WW} where the authors provide an exhaustive list of  references. There are a few results concerning the
 non-autonomous case, which  have been  recently obtained by Peng and Wang \cite{PW}, Pistoia and Vaira \cite{PV}
 and Li, Wei and Wu \cite{liweiwu}.\\

In the critical regime (i.e.  $n=4$) the existence of solutions to  \eqref{S} is a delicate and much more difficult issue. While there are some results concerning the  autonomous case obtained by Clapp and Pistoia \cite{CP}, Clapp and Szulkin \cite{CS} and  Chen, Medina and Pistoia \cite{CMP}, a very few is known about the non-autonomous one. As far as we know the first result is due to  Chen, Pistoia and Vaira \cite{CPV}, where system  \eqref{S}   is studied in a fully symmetric regime, namely   all the coupling parameters $\beta_{ij}$'s are equal to a real number $\beta$ and all the potentials $V_i$'s coincide with a positive and radially symmetric function  $V$, namely the system \eqref{S} reduces to the system
 \begin{equation}\label{beg}
-\Delta u_i+V(x)u_i=u_i^3+\beta \sum_{j\neq i}u_i u_j^2\text{~in~}{\mathbb R}^4, i=1,...,m.
\end{equation} In particular,   using this symmetric setting, if
\begin{equation}\label{ri} {\mathscr R}_i :=\(\begin{matrix}\cos \frac{2(i-1)\pi}{m } &\sin\frac{2(i-1)\pi}{m } &0&0\\
-\sin\frac{2(i-1)\pi}{m } &\cos\frac{2(i-1)\pi}{m } &0&0\\
0&0&1&0 \\
0&0&0 &1 \\
 \end{matrix}\right)\ \hbox{for any}\  i=1,...,m,\end{equation}
the authors build a solution $u_i(x)=u({\mathscr R}_i x)$ to the system \eqref{beg} via a solution $u$ to the {\em non-local} equation
\begin{align}\label{nlo}
-\Delta u+V(x)u=u^3+\beta u\sum_{i=2}^mu^2({\mathscr R}_i x)\ \text{~in~}{\mathbb R}^4.
\end{align}
Then, using a Ljapunov-Schmidt procedure and adapting the argument developed by Peng, Wang and Yan \cite{PWY}, they build a solution to \eqref{nlo} using the  {\em bubbles} 
  \begin{align}\label{udx}U_{\delta, \xi}(x)=\frac{1}{\delta} U\(\frac{x-\xi}{\delta}\)\ \hbox{with}\ U(x)=\frac{{\mathtt c}}{1+|x|^2},\ {{\mathtt c}}=2\sqrt{2},\end{align}
which  are all the positive solutions to the critical equation
 $$-\Delta u=u^{3} \text{~in~}{\mathbb R}^4.$$
 The   solution  to \eqref{nlo} they build looks like the sum of a large number $k$ of bubbles, i.e.  $\sum _{\ell=1}^k U_{\delta,\xi_\ell},$ 
 whose peaks $\xi_\ell$ are located at the vertices of a regular polygon placed in the sphere $\mathtt S(r_k):=\{(x_1,x_2,0,0)\ |\ x_1^2+x_2^2=r_k^2\}.$ Moreover,
  the radius $r_k$ as $k\to\infty$ approaches    a  non-degenerate critical point $r_0$ of the function $ r\to r^2V(r)$. 
 As a result, the solution to the system \eqref{beg}  are of {\em segregated} type in the sense that different components blow-up at different points as $k$ is large enough. Indeed, the shape of the function  $u_i$ resembles $k$ copies of bubbles whose    $k$ peaks are the points $\xi^i_\ell={({\mathscr R}_i )^{-1}}\xi_\ell,{ \ell=1,...,k}$ (see \eqref{ri}).

 The segregation phenomena  has been widely studied by Terracini and her collaborators in a series of papers 
(see for example \cite{CTV,29,37,38,39} and references therein) and naturally appears in the strongly repulsive case (i.e.  $\beta_{ij}\to -\infty$).
In \cite{CPV} the existence of this kind of solutions is not affected at all by the presence of the coupling parameter $\beta$: they do  exist in both fully repulsive (i.e. $\beta<0$) and attractive regime (i.e. $\beta>0$). This is due to the fact  that the regime of the system which is entirely encrypted in  the non-local term  of \eqref{nlo}  does not appear  in the main terms of the reduced problem.
Therefore, a natural question arises. \\
{\em (Q): \ do there exist any functions $V$  (possibly changing-sign)   such that existence of solutions to the system \eqref{beg} is affected by the sign of the parameter $\beta$?}
\\ 
In this paper, we give a positive answer.\\

We   assume that   $V$ is radially symmetric, $V\in  C^2({\mathbb R}^4)\cap {L^{2}(\mathbb R^4)}$, 
either $V\ge0$ or
$\|V\|_{L^{2}(\mathbb R^4)}\leq {4\over 3}\pi^2$
(see Remark \ref{rem1}), and
\begin{equation}\label{r0}
r_0>0\ \hbox{is a  non-degenerate critical point of  }\ r\to r^2V(r)\ \hbox{and}\ V(r_0)\not=0. 
\end{equation}
Our main result reads as follows.
\begin{theorem}\label{thm1.1}
There exists $m_0>0$ such that for any $m\geq m_0$, if $\beta= \mathfrak b{m^{-\alpha}}$ for some $\mathfrak b\not=0$ and $\alpha>0$ and if  $\mathfrak b V(r_0)>0,$ there exists a solution $u\in \mathcal D^{1,2}({\mathbb R}^4) $ to problem \eqref{nlo}  such that
as $m\to +\infty$,
\begin{equation*} u(x)\sim U_{\delta,  \zeta},\ \hbox{with}\ \zeta=(\rho,0,0,0),\ \rho\sim r_0\ \hbox{and}\  \delta\sim {\mathfrak d}r_0^2|V(r_0)|^{\frac12} \frac{1}{|\beta|^{\frac12}m^2}\end{equation*}  for some positive constant $\mathfrak d$.
\end{theorem}
According to the previous discussion,   as an immediate consequence of Theorem \ref{thm1.1} we get the following existence result of solutions for the system \eqref{beg}.
\begin{theorem}
There exists $m_0>0$ such that for any $m\geq m_0$, if $\beta= \mathfrak b{m^{-\alpha}}$ for some $\mathfrak b\not=0$ and $\alpha>0$ and if  $\mathfrak b V(r_0)>0,$ there exists a solution   $u_i(x)=u(\mathscr R_i x),$ $i=1,\dots,m$ to system \eqref{beg}, where $u$ is given in Theorem \ref{thm1.1}  and $\mathscr R_i $  is defined in  \eqref{ri}.
\end{theorem}

Let us make some comments.

\begin{remark}
If we want to find solutions depending on the sign of the parameter $\beta,$
  the coupling term $\beta u\sum_{i=2}^mu^2({\mathscr R}_ix)$ must be present  in the reduced problem
 \eqref{ex1}.
 To achieve this goal, we introduce the number $m$  of components as a large parameter. Unfortunately, the choice of a  large number of interaction
 as to be balanced by a suitable small parameter $\beta$ which depends on $m$. It would be extremely interesting to understand if this is a purely technical issue or not.
\end{remark}

\begin{remark}\label{rem1}
In order to build the solutions  in the competitive regime, i.e. $\beta<0$,  the function  $V$ has to be negative somewhere. In this case, we require that the linear operator
$-\Delta +V\mathtt {Id}$ is coercive and this is true if 
$$\|V\|_{L^2(\mathbb R^4)}< \min\limits_{u\in \mathcal D^{1,2}(\mathbb R^4)\setminus\{0\}}{\int\limits_{\mathbb R^4}|\nabla u|^2\over (\int\limits_{\mathbb R^4}|u|^4)^\frac12}=\frac43 \pi^2$$
as a direct computation shows (the minimum is achieved at the bubbles \eqref{udx}).
\end{remark}

\begin{remark}
The  system \eqref{S} in low dimensions, i.e. $n=2$ or $n=3$,
has been recently studied by Pistoia and Vaira in \cite{PV} in the same fully symmetric regime (see also  \cite{liweiwu}). They are lead to find solutions to the  non-local subcritical equation  \eqref{beg} defined in $\mathbb R^2$ or $\mathbb R^3$ (instead of $\mathbb R^4$). 
The existence of solutions in   the competitive (i.e. $\beta<0$) or cooperative regime (i.e. $\beta>0$) strongly depends on    the fact that   the radial potential $V$ has a maximum or a minimum at infinity, respectively. 
We believe that the idea of using the number of components as a large parameter  could also be applied in this setting to produce a new kind of solutions.
\end{remark}

 {\em Notations.} In the following we agree that $f\lesssim g$ or $f=\mathcal O(g)$ means $|f|\le C |g| $  for some positive constant $C$ independent of $m$ and $f\sim g$  means $f= g+o(g)$.

%
\section{Proof of Theorem \ref{thm1.1}}
\subsection{The ansatz }

We will find solutions of \eqref{nlo} in the space
\begin{align*} X:=\{u\in \mathcal D^{1,2}({\mathbb R}^4): u \text{~satisfies~} \eqref{sy22} \}\text{~with~}\|u\|:=\bigg(\int\limits_{\mathbb R^4}|\nabla u|^2dx\bigg)^{\frac12}\end{align*}
where
 \begin{align}
u(x_1, x_2, x_3, x_4)=u(x_1, -x_2, x_3, x_4)=u(x_1, x_2, -x_3, x_4)=u(x_1,x_2,x_3,-x_4).\label{sy22}
\end{align}

In particular, we are going to build a solution to \eqref{nlo} as
\begin{align*}u=\chi U_{\delta, \rho\xi_1} +\varphi,\end{align*}
where the  bubbles $U_{\delta, \xi}$ is defined in  \eqref{udx} whose blow-up point is
\begin{align*}
\rho \xi_1:= \rho(1,0,0,0),\ |\rho-r_0|\le \vartheta \end{align*}
for some small $\vartheta>0$ and  blow-up rate satisfies $\delta=\frac{d}{|\beta|^{\frac12}m^2}$with $d\in [d_1, d_2]$ for any fixed $0<d_1<d_2<+\infty$. Here $r_0$ is given in \eqref{r0}.
Moreover, $\chi(x)=\chi(|x|)$  is a  radial   cut-off function whose support is close to the sphere $|x|:=r=r_0$, namely 
  \begin{equation*}\chi=1 \text{~in~}|r-r_0|\leq \sigma\ \hbox{and}\ \chi=0\text{~in~}|r-r_0|>2\sigma\ \hbox{for some $\sigma>0$ small}\end{equation*}
 and $\varphi$ is an higher order term.
 
Let us point out that the function $u_i $ defined by $u_i(x)=u(\mathscr R_i x)$  (see \eqref{ri} )
  blows up at a single point
\begin{align*}\rho\xi_i:=\rho({\mathscr R}_i^{-1}\xi_1)= &\rho\bigg(\cos\(\frac{2(i-1)\pi}m\), \sin\(\frac{2(i-1)\pi}m\), 0, 0\bigg). \end{align*}
\subsection{Rewriting the non-local equation via the Lyapunov Schmidt reduction method}

Subtituting $u=\chi U_{\delta, \rho\xi_1}+\varphi$ into  the non-local problem \eqref{nlo}, it can be rewritten as
\begin{align}\label{len}
\mathcal L(\varphi)=\mathcal E+\mathcal N(\varphi) \text{~in~}{\mathbb R}^4\end{align}
where $\mathcal L(\varphi), \mathcal E, \mathcal N(\varphi)$ are defined as
\begin{align}
&\mathcal L(\varphi):=-\Delta\varphi+V(x)\varphi-3(\chi U_{\delta, \rho\xi_1})^2\varphi-\beta\varphi \sum_{i=2}^m(\chi U_{\delta, \rho\xi_i})^2, \label{l}\\
&\mathcal E:=(\chi U_{\delta, \rho\xi_1})^3+\Delta (\chi U_{\delta, \rho\xi_1})-V(x)\chi U_{\delta, \rho\xi_1}+\beta \chi^3U_{\delta, \rho\xi_1}\sum_{i=2}^mU^2_{\delta, \rho\xi_i}, \label{e}\\
&\mathcal N(\varphi):=\varphi^3+3\chi U_{\delta, \rho\xi_1}\varphi^2+\beta\varphi\sum_{i=2}^m\varphi^2(\mathscr R_ix)+2\beta \chi\varphi\sum_{i=2}^m U_{\delta, \rho\xi_i}\varphi(\mathscr R_ix)\nonumber\\
&~~~~~~~~~~~~~~~~~~~~~~~\ \ \ \ \ \ \ \ \ \ +\beta \chi U_{\delta, \rho\xi_1}\sum_{i=2}^m\varphi^2(\mathscr R_ix)+2\beta\chi^2 U_{\delta, \rho\xi_1}\sum_{i=2}^mU_{\delta, \rho\xi_i}\varphi(\mathscr R_ix).\label{n}
\end{align}
As it is usual, to solve \eqref{len} we will follow the classical steps of the Ljapunov-Schmidt procedure:
\begin{itemize}
\item[(i)] we show there  exists   $\varphi\in X$ solution to the problem
\begin{equation}\label{exi}\left\{\begin{aligned}
&\mathcal L(\varphi)=\mathcal E+\mathcal N(\varphi)+\sum_{l=0}^1{\mathfrak c}_l(\delta, \rho)(\chi U_{\delta, \rho\xi_1})^2Z_{\delta,\rho\xi_1}^l \text{~in~}\mathbb R^4, \\
& \int\limits_{\mathbb R^4}(\chi U_{\delta, \rho\xi_1})^2Z_{\delta,\rho\xi_1}^l\varphi \ dx=0, l=0, 1\end{aligned}\right.\end{equation}
where $$Z^0_{\delta, \rho\xi_1}=\frac{\partial (\chi U_{\delta, \rho\xi_1})}{\partial \delta}\ \hbox{and}\ Z^1_{\delta, \rho\xi_1}=\frac{\partial (\chi U_{\delta, \rho\xi_1})}{\partial \rho}.
$$
\item[(ii)] we find $\delta>0$ and $\rho$ close to $r_0$ such that   ${\mathfrak c}_0(\delta, \rho)={\mathfrak c}_1(\delta, \rho)=0$.
\end{itemize} 

At the beginning, we put forward two lemmas which are used frequently throughout the paper. The proofs can be founded by the similar arguments as Appendix A in \cite{WY} and Appendix A in \cite{MM} respectively.
\begin{lemma}\label{app1} For any $0<\tau\leq\min\{\tau_1, \tau_2\}, i\neq j,$ it holds
$$\begin{aligned}&\frac{1}{(1+|y-x_i|)^{\tau_1}}\frac{1}{(1+|y-x_j|)^{\tau_2}}\\
&~~~~~\lesssim \frac{1}{|x_i-x_j|^\tau}\(\frac{1}{(1+|y-x_i|)^{\tau_1+\tau_2-\tau}}+\frac{1}{(1+|y-x_j|)^{\tau_1+\tau_2-\tau}}\)\text{~for any~} y\in\mathbb R^4.\end{aligned}$$
\end{lemma}
\begin{lemma}
There exists a positive constant $C_\tau$ such that
\begin{equation}\label{tau}
 \sum_{i=2}^m\frac{1}{|\xi_1-\xi_i|^{\tau}} \sim C_\tau m^\tau,\ \text{~for any~}\tau>1
\end{equation}
for some positive constant $C_\tau.$
\end{lemma}

In the following proof, we will use the Jensen's inequality
\begin{equation*} 
\bigg(\frac1m\sum\limits_{i=1}^mt_i^p\bigg)^\frac1p\le \bigg(\frac1m\sum\limits_{i=1}^mt_i^q\bigg)^\frac1q\ \hbox{for any}\quad t_i\ge 0,\quad  0< p<q.\ 
\end{equation*}
When $p=1$ and $q>1$, it becomes
\begin{equation}\label{jen01}
\bigg(\sum\limits_{i=1}^{m}t_i\bigg)^q\le m^{q-1}\sum\limits_{i=1}^{m}t_i^q\ \hbox{for any}\ t_i\geqslant 0.
\end{equation}

Let us denote $\eta_i=\xi_i/\delta$ in the following.
\subsection{The size of the error term
 $\mathcal E$} 
\begin{proposition}\label{properr}Let $\mathcal E$ be defined as in \eqref{e}, then
$$\|\mathcal E\|_{L^{\frac{4}{3}}(\mathbb R^4)} \lesssim  \delta+|\beta|\delta^{2}m^{\frac{9}{4}} .$$
\end{proposition}
\begin{proof}
Notice that
\begin{align}\label{e123}\mathcal E=\underbrace{(\chi U_{\delta, \rho\xi_1})^3+\Delta (\chi U_{\delta, \rho\xi_1})}_{:=\mathcal E_1}-\underbrace{V(x)\chi U_{\delta, \rho\xi_1}}_{:=\mathcal E_2}+\underbrace{\beta \chi^3U_{\delta, \rho\xi_1}\sum_{i=2}^mU^2_{\delta, \rho\xi_i}}_{:=\mathcal E_3}.\end{align}
On one side, it gives from direct computations
\begin{equation*}\begin{aligned}\int\limits_{\mathbb R^4}\mathcal E_1^\frac43dx\leq \int\limits_{\mathbb R^4}\Big| \Big( \chi U_{\delta, \rho\xi_i}\Big)^3-\chi U_{\delta, \rho\xi_i}^3+2\nabla \chi\nabla U_{\delta,\rho\xi_1}+\Delta \chi U_{\delta, \rho\xi_1}\Big|^{\frac43}\lesssim \delta^{\frac43},\end{aligned}\end{equation*}
\begin{align*}
\int\limits_{\mathbb R^4}\mathcal E_2^\frac43dx\lesssim \int_{|r-r_0|\leq\sigma}  U_{\delta, \rho\xi_i}^{\frac43}dx\lesssim \delta^{\frac43}.\end{align*}
And then using \eqref{jen01} and Lemma \ref{app1}, one can get \begin{align}
\int\limits_{\mathbb R^4}\mathcal E_3^\frac43dx\lesssim&|\beta|^{\frac43}\int_{|r-r_0|\leq\sigma}\Big(\sum_{i=2}^m\frac{\delta}{\delta^2+|x-\rho\xi_1|^2}\frac{\delta^2}{(\delta^2+|x-\rho\xi_i|^2)^2}\Big)^{\frac43}dx\nonumber\\
\lesssim&|\beta|^{\frac43}\int_{|r-r_0|\leq\frac{\sigma}{\delta}}m^{\frac13}\sum_{i=2}^m\frac{1}{\rho^{\frac83}|\eta_1-\eta_i|^{\frac83}}\Big(\frac{1}{(1+|y-\rho\eta_1|)^{\frac{16}{3}}}+\frac{1}{(1+|y-\rho\eta_i|)^{\frac{16}{3}}}\Big)dy\nonumber\\
\lesssim&|\beta|^{\frac43}\delta^{\frac83}m^3.\label{ee3}
\end{align}
Therefore, it follows $\|\mathcal E\|_{L^{\frac43}(\mathbb R^4)}\lesssim \delta+|\beta|\delta^{2}m^{\frac{9}{4}}$.
\end{proof}
\subsection{The linear theory} 
\begin{proposition}\label{proplin}
Let  $\mathcal L(\varphi)$ be  defined as in \eqref{l} and $\beta=o(m^{-1})$. There exists a constant $m_0>0$ independent of $m$ such that for any $m\geq m_0$, $\delta=d \frac{1}{|\beta|^{\frac12}m^2}>0$ with $d\in [d_1, d_2]$ for any fixed $0<d_1< d_2<+\infty$, $\rho\in(r_0-\vartheta,r_0+\vartheta)$ with $\vartheta>0$ and for any $h\in L^{\frac43}({\mathbb R}^4)$ satisfying \eqref{sy22}, the linear problem
\begin{equation}\label{phi}\left\{\begin{aligned}&\mathcal L(\varphi)=h+\sum_{l=0}^1{\mathfrak c}_l (\chi U_{\delta, \rho\xi_1})^2Z_{\delta,\rho\xi_1}^l,\text{~in~} {\mathbb R}^4,\\&\int\limits_{\mathbb R^4} (\chi U_{\delta, \rho\xi_1})^2Z_{\delta,\rho\xi_1}^l\varphi \ dx=0, l=0,1\end{aligned}\right.\end{equation}
admits a unique solution $\varphi\in X$ satisfying 
\begin{align}\label{con}
\|\varphi\| \lesssim \|h\|_{L^{\frac43}(\mathbb R^4)}\text{~and~} {\mathfrak c}_l=\mathcal O(\delta \|h\|_{L^{\frac43}(\mathbb R^4)}), l=0,1
\end{align}
for some $\mathfrak c_0, \mathfrak c_1$.
\end{proposition}
\begin{proof}

\noindent \textbf{Step 1:} Assume first that \eqref{con}  holds. Define $$\tilde X:=\{\varphi\in X, \int\limits_{\mathbb R^4} (\chi U_{\delta, \rho\xi_1})^2Z_{\delta,\rho\xi_1}^l\varphi \ dx=0, l=0,1\}.$$
Denote the linear projection mapping $\Pi^{\perp}: X\to \tilde X$.

The first equation in \eqref{phi} can be rewritten as
\begin{align*}&\varphi+\underbrace{\Pi^{\perp}\big[(-\Delta)^{-1}(V(x)\varphi-3(\chi U_{\delta, \rho\xi_1})^2\varphi-\beta\varphi \sum_{i=2}^m(\chi U_{\delta, \rho\xi_i})^2)\big]}_{:=\mathcal K\varphi}\\
&\ \ \ =\underbrace{\Pi^{\perp}\big[(-\Delta)^{-1}(h+\sum_{l=0}^1{\mathfrak c}_l (\chi U_{\delta, \rho\xi_1})^2Z_{\delta,\rho\xi_1}^l)\big]}_{:=f}\end{align*}
where $\mathcal K: \tilde X\to \tilde X$ is a compact operator for each fixed $m$ since $U_{\delta, \rho\xi_1}^2, \sum\limits_{i=2}^mU_{\delta, \rho\xi_i}^2\in L^2({\mathbb R}^4), U_{\delta, \rho\xi_i}=\mathcal O(\frac{1}{|x|^2})$ if $|x|\to +\infty$, $V\in C^2(\mathbb R^{4})$ and $\|V\|_{L^{2}(\mathbb R^4)}$ is bounded if $V$ is non-negative otherwise $\|V\|_{L^{2}(\mathbb R^4)}\leq \frac43 \pi^2$ (see Lemma 2.3 in \cite{BS}) and there exist $\mathfrak c_0, \mathfrak c_1$ such that $f\in \tilde X$. Thus there is a unique $\varphi\in \tilde X$ such that $\varphi+\mathcal K\varphi=f$ with $f\in \tilde X$ by Fredholm-alternative theorem.
\vspace{1mm}

Now, what is left is to show that \eqref{con} holds ture. 
 Suppose that there exist $m_n\to +\infty$, $\varphi_{n}$ satisfying \eqref{phi} with $h=h_{n}$ such that $\|\varphi_{n}\|=1, \|h_{n}\|_{L^{\frac43}(\mathbb R^4)}\to 0$ as $n\to +\infty$. 

Letting $\tilde\varphi_n(x)=\delta_n\varphi_n(\delta_n x+\rho_n\xi_{1})$, up to a subsequence, we abtain 
$$\tilde\varphi_n\to \varphi^*\text{~weakly in~} \mathcal D^{1,2}(\mathbb R^4), \text{~strongly in~} L_{loc}^{p}(\mathbb R^4) \text{~for~} p\in [2, 4).$$
 \vspace{2mm}
  
 \noindent\textbf{Step 2:} Let us prove $\mathfrak c_l=\mathcal O(\delta_n\|\varphi_n\|)+\mathcal O(\delta_n \|h_n\|_{L^{\frac43}(\mathbb R^4)}), l=0,1$. For simplicity, we drop the subscript $n$ in this step. Multiplying the first equation in \eqref{phi} by $Z_{\delta, \rho\xi_1}^j, j=0, 1$, one can get
$$\begin{aligned}&\sum_{l=0}^1{\mathfrak c}_l\int\limits_{\mathbb R^4} (\chi U_{\delta, \rho\xi_i})^2Z^l_{\delta, \rho\xi_1}Z^j_{\delta, \rho\xi_1}dx\\ &=\int\limits_{\mathbb R^4}\(-\Delta\varphi+V(x)\varphi-3(\chi U_{\delta, \rho\xi_1})^2\varphi-\beta\varphi \sum_{i=2}^m(\chi U_{\delta, \rho\xi_i})^2-h\)Z^j_{\delta, \rho\xi_1}dx.\end{aligned}$$ 

Using straightforward computations and notice that $$\Big\|\frac{\partial U_{\delta, \rho\xi_i}}{\partial \delta}\Big\|_{L^{\frac43}(\mathbb R^4)}\leq C\text{~ and~} \Big\|\frac{\partial U_{\delta, \rho\xi_i}}{\partial \delta}\Big\|_{L^{4}(\mathbb R^4)}\lesssim \frac{1}{\delta},$$ we have for $j=0$($j=1$ is similar)
\begin{align*}
&\int\limits_{\mathbb R^4}(-\Delta\varphi-3(\chi U_{\delta, \rho\xi_1})^2\varphi)Z^0_{\delta, \rho\xi_1}dx\nonumber \\
&\lesssim\bigg|\int\limits_{\mathbb R^4}\Big(3\chi U_{\delta, \rho\xi_1}^2\frac{\partial U_{\delta, \rho\xi_1}}{\partial \delta}-3\chi^3U_{\delta, \rho\xi_1}^2\frac{\partial U_{\delta, \rho\xi_1}}{\partial \delta}\Big)\varphi dx\nonumber\\&\ \ \ \ \ \ \ \ +
\int\limits_{\mathbb R^4}(-\Delta \chi)\frac{\partial U_{\delta, \rho\xi_1}}{\partial\delta}\varphi dx-\int\limits_{\mathbb R^4}2\nabla \chi \nabla\frac{\partial U_{\delta, \rho\xi_1}}{\partial\delta}\varphi dx\bigg|\nonumber\\
&\lesssim \|\varphi\|,
\end{align*}
\begin{align*}
\int\limits_{\mathbb R^4}(V(x)\varphi -h)Z^j_{\delta, \rho\xi_1}dx&\lesssim \|\varphi\|\|Z^j_{\delta, \rho\xi_1}\|_{L^{\frac43}(\mathbb R^4)}+\|h\|_{L^{\frac43}(\mathbb R^4)}\|Z^j_{\delta, \rho\xi_1}\|_{L^{4}(\mathbb R^4)}\\&\lesssim \|\varphi\|+\frac{\|h\|_{L^{\frac43}(\mathbb R^4)}}{\delta}.
\end{align*}
Similar to \eqref{ee3}, one can deduce
\begin{align*}
\int\limits_{\mathbb R^4}\beta\varphi \sum_{i=2}^m(\chi U_{\delta, \rho\xi_i})^2Z^j_{\delta, \rho\xi_1}dx\lesssim&  |\beta|\|\varphi\|\sum_{i=2}^m\bigg( \int\limits_{\mathbb R^4}\Big(\chi^2 U_{\delta, \rho\xi_i}^2Z^j_{\delta, \rho\xi_1}\Big)^{\frac43}dx\bigg)^{\frac34}\\ \lesssim& |\beta|\|\varphi\|\sum_{i=2}^m\frac{1}{\delta}\bigg(\int_{|r-r_0|\leq\sigma}\Big(U_{\delta, \rho\xi_i}^2U_{\delta, \rho\xi_1}\Big)^{\frac43}dx\bigg)^{\frac34}\\ \lesssim&  |\beta|\|\varphi\|\delta m^2
\end{align*}
due to $|Z^j_{\delta, \rho\xi_1}|\lesssim \frac{U_{\delta, \rho\xi_1}}{\delta}$.
In addition, it is easy to check that there exist constants $B_l$ such that \begin{align*}
\int\limits_{\mathbb R^4} (\chi U_{\delta, \rho\xi_1})^2Z^l_{\delta, \rho\xi_1}Z^j_{\delta, \rho\xi_1}dx=\left\{\begin{aligned} &B_l \delta^{-2}(1+o(1)), l=j,\\& 0, \ \ \ \ \ \ \ \ \ \ \ \ \ \ \ \ \ \ \ l\neq j.\end{aligned}\right.
\end{align*}
So we have $\mathfrak c_l=\mathcal O((\delta^2+|\beta|\delta^3 m^{2})\|\varphi\|)+\mathcal O(\delta \|h\|_{L^{\frac43}(\mathbb R^4)})=\mathcal O(\delta(\|\varphi\|+ \|h\|_{L^{\frac43}(\mathbb R^4)})),l=0,1,$ which ends this step.

Thus $\varphi^*$ satisfies
 $$-\Delta\varphi^*=3U^2\varphi^*, \text{~in~}\mathbb R^4.$$
  
 \noindent\textbf{Step 3:} Let us prove $\varphi^*=0$. Denote $Z_i=\frac{\partial U}{\partial x_i}, i=1,...,4.$
 Recalling the second equation in \eqref{phi}, it suffices to show that \begin{equation*}I_{in}:=\int\limits_{ \mathbb R^4}\tilde{\varphi}_n U ^2 Z_idx=0,\quad i=1,...,4,\end{equation*}
Since $\varphi_n$ satisfies \eqref{sy22} and $\xi_1=(1, 0, 0,0)$, it follows $I_{2n}=I_{3n}=I_{4n}=0.$

Thanks to $\int\limits_{\mathbb R^4} (\chi U_{\delta_n, \rho_n\xi_{1}})^2Z_{\delta_n,\rho_n\xi_{1}}^1\varphi_n \ dx=0$, it derives $I_{in}=0, i=1,...,4.$
\vspace{2mm}

\noindent\textbf{Step 4:} Let us prove $\int\limits_{\mathbb R^4}|\nabla\varphi_n|^2dx\to 0$ as $n\to +\infty$. 

Firstly, since $\|V\|_{L^{2}(\mathbb R^4)}$ is bounded if $V$ is non-negative, otherwise $\|V\|_{L^{2}(\mathbb R^4)}\leq \frac43 \pi^2$,
we have $\|\varphi\|_V:=(\int\limits_{\mathbb R^4}|\nabla \varphi|^2+V(x)\varphi^2dx)^{\frac12}\lesssim \|\varphi\|\lesssim \|\varphi\|_V$. 

Testing the first equation in \eqref{phi} by $\varphi_n$,  one deduces from the second equation in \eqref{phi},
\begin{align*}
& \int\limits_{\mathbb R^4}|\nabla\varphi_n|^2+ V(x)\varphi_n^2dx\\
&=\int\limits_{\mathbb R^4}3\varphi_n^2(\chi U_{\delta_n, \rho_n\xi_1})^2dx+\int\limits_{\mathbb R^4}\beta \sum_{i=2}^m(\chi U_{\delta_n, \rho_n\xi_{in}})^2 \varphi_n^2dx+\int\limits_{\mathbb R^4}h_n\varphi_n dx.
\end{align*}
Since $\tilde\varphi_n\to 0$ weakly in $L^4(\mathbb R^4)$, it follows
\begin{align*}
\int\limits_{\mathbb R^4}3\varphi_n^2(\chi U_{\delta_n, \rho_n\xi_1})^2dx \sim 3\int\limits_{\mathbb R^4}\tilde\varphi_n^2 U^2dx\to 0.
\end{align*}
Reminding that $\beta=o(m^{-1})$, one can get
\begin{align*}
\beta\int\limits_{\mathbb R^4}\sum_{i=2}^m(\chi U_{\delta_n, \rho_n\xi_{in}})^2 \varphi_n^2dx\leq |\beta|\cdot m\|U_{\delta_n, \rho_n\xi_{in}}\|^2_{L^{4}(\mathbb R^4)}\|\varphi_n\|^2\to 0,
\end{align*}
and according to $\|h_n\|_{L^{\frac43}(\mathbb R^4)}\to 0, \|\varphi_n\|=1$, it yields $$\int\limits_{\mathbb R^4}h_n\varphi_n dx\to 0,$$
i.e.,  $\int\limits_{\mathbb R^4}|\nabla\varphi_n|^2+V(x)\varphi_n^2dx\to 0$ and then $\|\varphi_n\|\to 0$ as $n\to +\infty$, which contradicts with $\|\varphi_n\|=1$.
 \end{proof}
 
We are going to solve the non-linear problem \eqref{exi} by using the standard fixed point theorem. 
\begin{proposition}
Assume $\beta=o(m^{-1})$, there exists a constant $m_0>0$ independent of $m$ such that for any $m\geq m_0$, $\delta=d \frac{1}{|\beta|^{\frac12}m^2}>0$ with $d\in [d_1, d_2]$ for any fixed $0<d_1< d_2<+\infty$ and $\rho\in(r_0-\vartheta,r_0+\vartheta)$ with $\vartheta>0$, the problem \eqref{exi} has a unique solution $\varphi\in \tilde X$ satisfying
$$\|\varphi\|\lesssim \delta+\frac{1}{m^{\frac74}},\ \mathfrak c_l=o(1), l=0,1.$$ 
\end{proposition}
\begin{proof}
The proof depends on a standard contraction mapping argument together with Proposition \ref{properr} and then $$\|\varphi\|\lesssim \|\mathcal E\|_{L^{\frac43}(\mathbb R^4)}\lesssim \delta+|\beta|\delta^{2}m^{\frac{9}{4}}\lesssim \delta+\frac{1}{m^{\frac74}}.$$
\end{proof}

%
We are going to choose $\delta>0, \rho\to r_0$ such that $\mathfrak c_0=\mathfrak c_1=0$ by using the following two propositions.
\subsection{The reduced problem}
 \begin{proposition} Let $\beta=o(m^{-1}), \delta=o(m^{-1})$. There exists a positive constant $\mathfrak b$ such that
\begin{align}\int\limits_{{\mathbb R}^4}(\mathcal L(\varphi)-\mathcal E-\mathcal N(\varphi))Z^0_{\delta, \rho\xi_1}&dx=-8V(\rho)\delta(\ln \delta) +\mathfrak b\beta\delta^3(\ln\delta m) \frac{ m^4}{\rho^4}
+\mathcal O(\delta+\frac{|\beta|m\|\varphi\|^2}{\delta}\nonumber\\&+\frac{\|\varphi\|^2}{\delta}+|\beta|\|\varphi\|^2\delta|\ln\delta|^{\frac12}m^{2}+|\beta|\|\varphi\|\delta m^{2}+\|\varphi\|)\label{ex1}
\end{align}
where $\mathcal L(\varphi), \mathcal E, \mathcal N(\varphi)$ are defined in \eqref{l}-\eqref{n}.
\end{proposition}
\begin{proof}
Notice that 
\begin{align}& \int\limits_{{\mathbb R}^4}(\mathcal L(\varphi)-\mathcal E-\mathcal N(\varphi))Z^0_{\delta, \rho\xi_1}dx\nonumber\\&=-\underbrace{\int\limits_{{\mathbb R}^4}\mathcal E_1Z^0_{\delta, \rho\xi_1}dx}_{:=\mathcal Q_1}+\underbrace{\int\limits_{{\mathbb R}^4}\mathcal E_2Z^0_{\delta, \rho\xi_1}dx}_{:=\mathcal Q_2}-\underbrace{\int\limits_{\mathbb R^4}\mathcal E_3Z^0_{\delta, \rho\xi_1} dx}_{:=\mathcal Q_3}\nonumber\\&\ \ \ -\underbrace{\int\limits_{\mathbb R^4} \mathcal N(\varphi)Z^0_{\delta, \rho\xi_1} dx}_{:=\mathcal Q_4}+\underbrace{\int\limits_{\mathbb R^4} \mathcal L(\varphi)Z^0_{\delta, \rho\xi_1} dx}_{:=\mathcal Q_5}
\label{cla}\end{align}
where $\mathcal E_1, \mathcal E_2, \mathcal E_3$ are defined in \eqref{e123}.

Let us estimate $\mathcal Q_1$. it is immediate to get
\begin{align}\label{q1}
\mathcal Q_1=\int_{\mathbb R^4}\bigg[\Big( \chi U_{\delta, \rho\xi_i}\Big)^3-\chi U_{\delta, \rho\xi_i}^3+2\nabla \chi\nabla U_{\delta,\rho\xi_1}+\Delta \chi U_{\delta, \rho\xi_1}\bigg]Z^0_{\delta, \rho\xi_1}dx \lesssim \delta.
\end{align}

Let us estimate $\mathcal Q_2$. We obtain
\begin{align}
\mathcal Q_2
=&V(\rho)\int_{|r-r_0|\leq\sigma}U_{\delta, \rho\xi_1}\frac{\partial U_{\delta, \rho\xi_1}}{\partial \delta} dx+\mathcal O(\delta)\nonumber\\
=&V(\rho)\delta\int_{B(0, {\frac{\sigma}{\delta}})}-U(\langle y, \nabla U\rangle +U)dy+\mathcal O(\delta)\nonumber\\
=&V(\rho)\delta\int_{B(0, {\frac{\sigma}{\delta}})}U^2dy+\mathcal O(\delta)\nonumber\\
\sim&-8V(\rho)\delta\ln\delta\label{vf}
\end{align}
where we use ${{\mathtt c}}^2=8>0$ (see \eqref{udx}).

Let us estimate $\mathcal Q_3$. We claim
\begin{align}
&\int\limits_{\mathbb R^4}\mathcal E_3Z^0_{\delta, \rho\xi_1} dx\nonumber\\
&\sim \beta\sum_{i=2}^m \int_{|r-r_0|\leq \sigma}U_{\delta, \rho\xi_1}\frac{\partial U_{\delta, \rho\xi_1}}{\partial \delta}U^2_{\delta, \rho\xi_i}dx\nonumber\\
&\sim\beta\sum_{i=2}^m\underbrace{\int\limits_{|x-\rho\xi_1|\leq \frac{\rho|\xi_1-\xi_i|}{2}}\dots}_{I_1}+\underbrace{\int\limits_{|x-\rho\xi_i|\leq \frac{\rho|\xi_1-\xi_i|}{2}}\dots}_{I_2}+\underbrace{\int\limits_{|x-\rho\xi_1|\geq \frac{\rho|\xi_1-\xi_i|}{2}\atop
\frac{\rho|\xi_1-\xi_i|}{2}\leq |x-\rho\xi_i|\leq 2\rho|\xi_1-\xi_i|
}\dots }_{I_3}+\underbrace{\int\limits_{|x-\rho\xi_1|\ge \frac{\rho|\xi_1-\xi_i|}{2}\atop
|x-\rho\xi_i|\ge 2\rho|\xi_1-\xi_i|
}\dots }_{I_4}\nonumber\\
&\sim 2\beta\sum_{i=2}^m\frac{{\mathtt c}^4\delta^3}{\rho^4|\xi_1-\xi_i|^4}\ln\frac{\rho|\xi_1-\xi_i|}{2\delta}+\mathcal O\(\sum_{i=2}^m\frac{\delta^3|\beta|}{\rho^4|\xi_1-\xi_i|^4}\)\nonumber\\
&\sim -\mathfrak b\beta\delta^3(\ln\delta m)\frac{ m^4}{\rho^4}\label{main}
\end{align}
where $\mathfrak b=2{\mathtt c}^4C_4$(see \eqref{tau}), since $\sum\limits_{i=2}^m\frac{1}{|\xi_1-\xi_i|^4}\ln(m|\xi_1-\xi_i|)=\mathcal O(m^4)$.

Indeed, on one hand, by the Lagrange Mean Value Theorem,  for any $y\in B(0, \frac{\rho|\xi_i-\xi_j|}{2\delta }), \xi_i\neq \xi_j$, we have
\begin{align*}& U_{\delta, \rho\xi_j}(\delta y+\rho\xi_i)= {\mathtt c} \frac{\delta}{\rho^2|\xi_j-\xi_i|^2}\bigg[1-\frac{2\delta\rho\langle y, \xi_i-\xi_j\rangle}{\rho^2|\xi_i-\xi_j|^2}+\mathcal O\bigg(\frac{\delta^2(1+|y|^2)}{\rho^2|\xi_j-\xi_i|^2}\bigg)\bigg]\text{~uniformly},\end{align*}
so \begin{align}
I_1&\sim  {\mathtt c}^2\delta^2\frac{1}{\rho^4|\xi_i-\xi_1|^4} \int_{B(\rho\xi_1,\frac{\rho|\xi_1-\xi_i|}{2})}U_{\delta, \rho\xi_1}\frac{\partial U_{\delta, \rho\xi_1}}{\partial \delta} dx\nonumber\\
&\sim  {\mathtt c}^2\delta^3\frac{1}{\rho^4|\xi_i-\xi_1|^4}\int_{B(0,\frac{\rho|\xi_1-\xi_i|}{2\delta})}U^2 dy\nonumber\\
&\sim \frac{{\mathtt c}^4\delta^3}{\rho^4|\xi_1-\xi_i|^4}\ln\frac{\rho|\xi_1-\xi_i|}{2\delta}.\label{mai3}
\end{align}
 And one can check that \begin{align}
I_2&={\mathtt c}^4\int\limits_{|x-\rho\xi_i|\le \frac{\rho|\xi_i-\xi_1|}2}\delta^3{|x-\rho\xi_1|^2-\delta^2\over\(|x-\rho\xi_1|^2+\delta^2\)^3}{1\over\(|x-\rho\xi_i|^2+\delta^2\)^2}dx\nonumber\\
&\hbox{set}\ x=\delta y+\rho\xi_i\nonumber\\
&={\mathtt c}^4\int\limits_{|y|\le \frac{\rho|\xi_i-\xi_1|}{2\delta}}\delta^3{|\delta y+\rho\xi_i-\rho\xi_1|^2-\delta^2\over\(|\delta y+\rho\xi_i-\rho\xi_1|^2+\delta^2\)^3}{1\over\(|y|^2+1\)^2}dy\nonumber\\
&={\mathtt c}^4\int\limits_{|y|\le \frac{\rho|\xi_i-\xi_1|}{2\delta}}\delta^3{ 1\over\(|y|^2+1\)^2}{1\over | \rho\xi_1-\rho\xi_i|^4}dy\nonumber\\ &+
{\mathtt c}^4 \int\limits_{|y|\le \frac{\rho|\xi_i-\xi_1|}{2\delta}}\delta^3{ 1\over\(|y|^2+1\)^2}\({|\delta y+\rho\xi_i-\rho\xi_1|^2-\delta^2\over\(|\delta y+\rho\xi_i-\rho\xi_1|^2+\delta^2\)^3}-{1\over |\rho\xi_1-\rho\xi_i|^4}\)dy\nonumber\\
&\sim\frac{{\mathtt c}^4\delta^3}{\rho^4|\xi_1-\xi_i|^4}\ln\frac{\rho|\xi_1-\xi_i|}{2\delta}+\mathcal O\Big(\frac{\delta^3}{\rho^4|\xi_1-\xi_i|^4}\Big)\label{i2}
\end{align}
since  it holds from $|\theta\delta y+\rho\xi_i-\rho\xi_1|\geq c\rho|\xi_1-\xi_i|$ for some $\theta\in(0, 1), c>0$,
\begin{align*}&{|\delta y+\rho\xi_i-\rho\xi_1|^2-\delta^2\over\(|\delta y+\rho\xi_1-\rho\xi_i|^2+\delta^2\)^2}-{1\over |\rho\xi_1-\rho\xi_i|^4}\\&= \delta {\theta\delta+\langle\theta\delta y+\rho\xi_1-\rho\xi_i,y\rangle\over\(|\theta\delta y+\rho\xi_1-\rho\xi_i|^2+(\theta\delta)^2\)^3}\\
 &\lesssim {\delta^2\over(\rho|\xi_1-\xi_i|)^6}+{\delta^2|y|^2+\delta\rho|\xi_1-\xi_i||y|\over(\rho|\xi_1-\xi_i|)^6}.\end{align*}

On the other hand, it shows from direct calculations that
\begin{align}
|I_3|\lesssim& \frac{\delta}{\rho^4|\xi_1-\xi_i|^4}\int\limits_{
\frac{\rho|\xi_1-\xi_i|}{2}\leq |x-\rho\xi_i|\leq 2\rho|\xi_1-\xi_i|
}\frac{\delta^2}{(\delta^2+|x-\rho\xi_i|^2)^2}dx\nonumber\\
\lesssim& \frac{\delta^3}{\rho^4|\xi_1-\xi_i|^4}\Big(\ln \frac{2\rho|\xi_1-\xi_i|}{\delta}-\ln\frac{\rho|\xi_1-\xi_i|}{2\delta}\Big)\nonumber\\
\lesssim&\frac{\delta^3}{\rho^4|\xi_1-\xi_i|^4}\label{i3}\end{align}
and \begin{align}|I_4|
&\lesssim \int\limits_{|x-\rho\xi_1|\ge \frac{|\rho\xi_1-\rho\xi_i|}{2}\atop
|x-\rho\xi_i|\ge 2|\rho\xi_1-\rho\xi_i|
}\delta^3{1\over |x-\rho\xi_1|^4}{1\over |x-\rho\xi_i|^4}dx\nonumber\\
&\hbox{set}\ x=\rho|\xi_1-\xi_i| y+\rho\xi_i\nonumber\\
&\lesssim   \int\limits_{|y|\ge 2\atop
|y+(\xi_1-\xi_i)/|\xi_1-\xi_i||\ge \frac{|y|}{2}}\delta^3{1\over |y|^4}{1\over  |y+(\xi_1-\xi_i)/|\xi_1-\xi_i||^4}dy\nonumber\\
&\lesssim \delta^3.\label{i4}\end{align}
Combining \eqref{mai3}-\eqref{i4}, then \eqref{main} as desired.

Let us estimate $\mathcal Q_4$. Indeed,
\begin{align}
\int\limits_{\mathbb R^4}\mathcal N(\varphi)Z^0_{\delta, \rho\xi_1}dx\lesssim&\frac{\|\varphi\|^3}{\delta}+\frac{|\beta|m\|\varphi\|^3}{\delta}+\frac{\|\varphi\|^2}{\delta}+ \frac{|\beta|m\|\varphi\|^2}{\delta}\nonumber\\
&+|\beta|\|\varphi\|^2\sum_{i=2}^m\bigg(\int\limits_{\mathbb R^4}\Big(\chi U_{\delta, \rho\xi_i}Z^0_{\delta, \rho\xi_1}\Big)^2dx\bigg)^{\frac12}\nonumber\\
&+|\beta|\|\varphi\|\sum_{i=2}^m \bigg(\int\limits_{\mathbb R^4}\Big(\chi^2 U_{\delta, \rho\xi_1} U_{\delta, \rho\xi_i}Z^0_{\delta, \rho\xi_1}\Big)^{\frac43}dx\bigg)^{\frac34}\nonumber \\
\lesssim &\frac{|\beta|m\|\varphi\|^2}{\delta}+\frac{\|\varphi\|^2}{\delta}+|\beta|\|\varphi\|^2\delta|\ln\delta |^{\frac12}m^{2}+|\beta|\|\varphi\|\delta m^{2}\label{ne}
\end{align}
since from \eqref{jen01} and Lemma \ref{app1}
\begin{align*}
&\sum_{i=2}^m\bigg(\int\limits_{\mathbb R^4}\Big(\chi U_{\delta, \rho\xi_i}Z^0_{\delta, \rho\xi_1}\Big)^2dx\bigg)^{\frac12}\\
&\lesssim\sum_{i=2}^m \bigg(\frac{1}{\delta^2}\int_{|r-r_0|\leq\frac{\sigma}{\delta}}\frac{1}{\rho^4|\eta_1-\eta_i|^4}\Big(\frac{1}{(1+|y-\rho\eta_1|)^4}+\frac{1}{(1+|y-\rho\eta_i|)^4}\Big) dy\bigg)^{\frac12}\\
&\lesssim \delta m^2|\ln\delta|^{\frac12}
\end{align*}
and 
\begin{align*}
&\sum_{i=2}^m\bigg(\int\limits_{\mathbb R^4}\Big(\chi^2 U_{\delta, \rho\xi_1} U_{\delta, \rho\xi_i}Z^0_{\delta, \rho\xi_1}\Big)^{\frac43}dx\bigg)^{\frac34}\\
&\lesssim\sum_{i=2}^m\bigg(\frac{1}{\delta^{\frac43}}\int_{|r-r_0|\leq\frac{\sigma}{\delta}}\frac{1}{\rho^{\frac83}|\eta_1-\eta_i|^{\frac83}}\Big(\frac{1}{(1+|y-\rho\eta_1)^{\frac{16}{3}}}+\frac{1}{(1+|y-\rho\eta_i)^{\frac{16}{3}}}\Big)dy\bigg)^{\frac34}\nonumber\\
&\lesssim\delta m^2.
\end{align*}

Let us estimate $\mathcal Q_5$. We derive from step 2 in Proposition \ref{proplin}
\begin{align}\label{le}
&\int\limits_{\mathbb R^4}(-\Delta\varphi+V(x)\varphi-3(\chi U_{\delta, \rho\xi_1})^2\varphi-\beta\varphi \sum_{i=2}^m(\chi U_{\delta, \rho\xi_i})^2)Z^0_{\delta, \rho\xi_1}dx
\lesssim \|\varphi\|+ |\beta|\|\varphi\|\delta m^{2}.
\end{align}
Therefore one can end with \eqref{cla}-\eqref{main}, \eqref{ne}, \eqref{le} immediately.
 \end{proof}

 \begin{proposition}Denote $D_{\varepsilon}:=\{|r-r_0|\leq \varepsilon, \varepsilon\in (2\sigma, 5\sigma)\}$. It holds true that
 \begin{align}
&\int\limits_{D_\varepsilon}(\mathcal L(\varphi)-\mathcal E-\mathcal N(\varphi))\langle x, \nabla u\rangle dx\nonumber\\
&\ \ \ \ \ \ =\frac{1}{\rho  }\frac{\partial ( \rho  ^2V( \rho))}{\partial  \rho  }4 \delta^2\ln\delta+\mathcal O(\delta^2+\|\varphi\|^2+|\beta|m\|\varphi\|^4)\label{pooo2}
 \end{align}
where $\mathcal L(\varphi), \mathcal E, \mathcal N(\varphi)$ are defined in \eqref{l}-\eqref{n}.
 \end{proposition}
 \begin{proof}
By the similar arguments as Proposition 2.10 of \cite{CPV}, we immediately get
 \begin{align}&\int\limits_{D_\varepsilon}(-\Delta u+V(x)u-u^3-\beta u \sum_{i=2}^m u^2(\mathscr R_ix))\langle x, \nabla u\rangle dx\nonumber\\
&=\int\limits_{D_\varepsilon}(-\Delta u+V(x)u-u^3)\langle x, \nabla u\rangle dx +\int\limits_{D_\varepsilon}\beta u^2\sum_{i=2}^mu^2(\mathscr R_ix)dx+\mathcal O(\beta m \|\varphi\|^4)\nonumber\\
 &=\int\limits_{D_{\varepsilon}}-|\nabla u|^2- V(x)u^2-(V(x) + \frac12\langle \nabla V,  x\rangle)u^2+u^4\nonumber\\
 &\ \ \ \ \ \ \ \ \ \ \ +\beta u^2\sum_{i=2}^mu^2(\mathscr R_ix)dx+\mathcal O(\|\varphi\|^2+\beta m\|\varphi\|^4)\nonumber\\
&=\int\limits_{D_\varepsilon}-(V(x) + \frac12\langle \nabla V(x),  x\rangle)u^2dx+\mathcal O(\|\varphi\|^2+\beta m\|\varphi\|^4)+o(\delta^2)\nonumber\\
&=\int\limits_{D_\varepsilon}-\frac{1}{2r}\frac{\partial (r^2V(r))}{\partial r}u^2dx+\mathcal O(\|\varphi\|^2+\beta m\|\varphi\|^4)+o(\delta^2)\nonumber\\
&=\frac{1}{\rho}\frac{\partial (\rho^2V(\rho))}{\partial \rho} 4\delta^2\ln\delta +\mathcal O(\delta^2+\|\varphi\|^2+\beta m\|\varphi\|^4) \nonumber
\end{align}
where we use $\frac12{\mathtt c}^2=4$. Thus we are done.
 \end{proof}

 \subsection{Proof of Theorem \ref{thm1.1}: completed}
Similar to Proposition 2.8 in \cite{CPV}, taking advantage of \eqref{ex1} and \eqref{pooo2}, the problem reduces to find $(\delta, \rho)$ such that 
$$\left\{ \begin{aligned}
 &-8\delta\ln\delta  V(\rho)+\mathfrak b\beta\delta^3\ln (\delta m)\frac{ m^4}{\rho^4} =\mathcal O\bigg(\delta+\|\varphi\|^2\Big(\frac{|\beta|m}{\delta}+\frac{1}{\delta}\\
 &\ \ \ \ \ \ \ \ \ \ \ \ \ \ \ \ \ \ \ \ \ \ \ \ \ \ \ \ \ \ \ \ \ \ \ +|\beta|\delta|\ln\delta|^{\frac12}m^{2}\Big)+\|\varphi\|(|\beta|\delta m^{2}+1)\bigg),\\
 &4\delta^2\ln\delta  \frac{1}{\rho}\frac{\partial(\rho ^2V(\rho))}{\partial \rho } =\mathcal O(\delta^2+\|\varphi\|^2+\beta m\|\varphi\|^4).
 \end{aligned}\right.$$
 Let us choose $\beta\sim\frac{\mathfrak e}{m^{\alpha}} $ with $ 1<\alpha<2$ for some constant $\mathfrak e$ and $\delta^2= \frac{\tilde d_m}{\beta m^4}$ for some $\tilde d_m\neq 0 $, we have $\|\varphi\|\lesssim  \frac{1}{m^{\frac74}}+\delta\lesssim \delta$, then it can be verified that
\begin{equation*}\left\{\begin{aligned}&\mathcal O\bigg(\delta+\|\varphi\|^2\Big(\frac{|\beta|m}{\delta}+\frac{1}{\delta}+|\beta|\delta|\ln\delta|^{\frac12}m^{2}\Big)+\|\varphi\|(|\beta|\delta m^{2}+1)\bigg)=o(\delta|\ln\delta|),\\
&\mathcal O(\delta^2+\|\varphi\|^2+\beta m\|\varphi\|^4)=o(\delta^2|\ln\delta|).
\end{aligned}\right.\end{equation*}

Since $V(r_0)$ and $\beta$ have the same sign, it is equivalent to find $\tilde d=\tilde d_m\neq 0$  and $\rho=\rho_m>0$
 solutions of
 \begin{equation}\label{last}
\left\{ \begin{aligned}
&\tilde{\mathfrak a}V(\rho) -\frac{\tilde{\mathfrak b} \tilde d}{\rho^4}=o(1),\\
& \frac{\partial(\rho^2V(\rho))}{\partial \rho}=o(1).
\end{aligned}\right.\end{equation}
In the end, since $r_0$ is a non-degenerate critical point of the function $r^2V(r)$ and $V(r_0)\neq 0$,
the problem \eqref{last} has a solution $(\tilde d_m, \rho_m)$ and $\sim \frac{\tilde{\mathfrak a}r^4_0V(r_0)}{\tilde{\mathfrak b}}\neq 0,
 \rho_m\sim r_0$ as $m$ is large enough. Taking $d=\sqrt{|\tilde d_m|}$, the proof is completed.
 \qed

\end{document}